\documentclass[a4paper, british, oneside]{amsart}

\usepackage[utf8]{inputenc}
\usepackage[T2A]{fontenc}

\usepackage[english]{babel}

\usepackage{amsmath,amsthm,amssymb}
\usepackage{geometry}
\usepackage[hidelinks,backref]{hyperref}

\newcommand{\Z}{\mathbb{Z}}
\newcommand{\N}{\mathbb{N}}
\newcommand{\F}{\mathbb{F}}
\newcommand{\Q}{\mathbb{Q}}
\newcommand{\K}{\mathbb{K}}
\DeclareMathOperator{\kar}{char}

\DeclareMathOperator{\lcm}{lcm}
\newcommand{\Lring}{\mathcal{L}_{\mathrm{ring}}}

\newcommand{\pow}[1]{(\!(#1)\!)}

\theoremstyle{definition}
\newtheorem{theorem}{Theorem}[section]
\newtheorem{lemma}[theorem]{Lemma}
\newtheorem{proposition}[theorem]{Proposition}
\newtheorem{corollary}[theorem]{Corollary}

\theoremstyle{remark}
\newtheorem{remark}[theorem]{Remark}
\newtheorem{remarks}[theorem]{Remarks}

\author{Philip Dittmann}
\address{Institut für Algebra, Technische Universität Dresden, 01062 Dresden, Germany}
\email{philip.dittmann@tu-dresden.de}

\title{Two examples concerning existential undecidability in fields}

\begin{document}
\maketitle

\section{Introduction}

Given a field $K$, one may ask whether there is an algorithm to decide which multivariable polynomials with coefficients in the prime field have zeroes in $K$ -- in short, whether $K$ is \emph{existentially decidable}.
Motivated by Hilbert's Tenth Problem, much research has been done on this question in particular in global fields and function fields, see for instance the monograph \cite{Shlapentokh_book}.
On the other hand, this question is also of interest in henselian valued fields, where it is the first step of a good model-theoretic understanding of the full first-order theory.
See in particular \cite{AnscombeFehm_existential-old, AnscombeJahnke_Cohen, ADF_existential, Kartas_tame} for recent related work.

The chief aim of this note is to prove the following theorem, giving an interesting example of existential undecidability.
\begin{theorem}\label{thm:intro-valued}
  Let $p$ be a prime number.
  There exists a complete discretely valued field $(E,v)$ of characteristic $0$ and residue characteristic $p$ such that the residue field $Ev$ is existentially decidable, the set of polynomials in $\Q[X]$ with a zero in $E$ is decidable, but the field $E$ is existentially undecidable.
\end{theorem}
This answers a question by Anscombe--Fehm in a strong way, see Remark \ref{rem:question-AF} for a discussion.

In order to prove this theorem, we use an example of a different phenomenon in existential decidability, which seems interesting in its own right.
\begin{theorem}\label{thm:intro-Kesavan}
  Let $p$ be a prime number.
  There exists an existentially decidable field of characteristic $p$ with an existentially undecidable quadratic extension.
\end{theorem}
A variant of this problem was first considered in Kesavan Thanagopal's thesis \cite{Kesavan}, where an example was given in characteristic $0$.
We modify the construction given there, based on Ershov's theory of fields with a strong local-global principle presented in \cite{Ershov}.

\subsection*{Acknowledgements}

I became aware of the examples presented here some years ago.
I would like to thank Arno Fehm for encouraging me to commit them to writing, as well as for comments on a draft version.

\section{A useful family of varieties}

Let $p$ be a prime number, $q>1$ a power of $p$.
In this section we prove the following proposition, which will be useful later.
\begin{proposition}\label{prop:testVariety}
  Let $n \geq 1$.
  There exists a smooth projective geometrically integral variety $V/\F_q$ such that for any $m \geq 1$ we have:
  \begin{itemize}
  \item If $m \mid n$, then $V(\F_{q^m}) = \emptyset$;
  \item if $\lcm(m, n) \geq 4n$, then $V(\F_{q^m}) \neq \emptyset$.
  \end{itemize}
\end{proposition}
For definiteness, in this article we take a \emph{variety} (over a specified base field) to be a separated scheme of finite type, although almost all varieties occurring will be quasi-projective and geometrically integral.

The proof of Proposition \ref{prop:testVariety} relies on the following lemma.
\begin{lemma}
  There exists a smooth projective geometrically integral curve $C/\F_q$ such that $C(\F_q) = \emptyset$, but $C(k) \neq \emptyset$ for any field extension $k/\F_q$ with $4 \leq [k : \F_q] < \infty$.
\end{lemma}
\begin{proof}
  Let $g$ be the smallest integer bigger than $\frac{q-3}{2}$ with $g \equiv -1 \pmod{p}$, so $\frac{q-3}{2} < g \leq \frac{q-3}{2} + p$.
  By \cite[Lemma 2.2]{BeckerGlass_pointless-hyperelliptic}, there exists a hyperelliptic curve $C/\F_q$ of genus $g$ with $C(\F_q) = \emptyset$.

  The number of $\F_{q^m}$-rational points of $C$ is at least $q^m + 1 - 2g\sqrt{q^m} \geq q^m + 1 - (q-3 + 2p)q^{m/2}$ by the Hasse--Weil bound.
  This is positive if $q-3 + 2p \leq q^{m/2}$, which is the case if $m \geq 4$.
\end{proof}

\begin{remark}
  The situation would be neater if we could strengthen the lemma to say that $C(k) \neq \F_q$ for any proper finite extension $k/\F_q$, in which case we could also strengthen the proposition to say that $V(\F_{q^m}) = \emptyset$ if and only if $m \mid n$.

  In order to improve the lemma in this way, one would need to improve the construction of Becker and Glass, finding a bound for the genus which is better than linear in $q$.
This works at least for specific values for $q$ in any characteristic $p > 3$, see \cite{Yekhanin_plane-pointless-curves}.
\end{remark}

\begin{proof}[Proof of the proposition]
  Let $C/\F_{q^n}$ be a curve as in the lemma, so that $C(\F_{q^n}) = \emptyset$ but $C(\F_{q^{nl}}) \neq \emptyset$ for $l \geq 4$.
  Let $V/\F_q$ be the Weil restriction of $C$.
  It is a smooth projective geometrically integral variety over $\F_q$ because $C/\F_{q^n}$ is so:
  Indeed, by \cite[Proposition A.5.9]{pseudoreductive} $V$ is smooth and geometrically connected (hence geometrically integral), and by \cite[Proposition A.5.8]{pseudoreductive} and \cite[Proposition 7.6/5]{BLR} $V$ is quasi-projective and proper, hence projective.

  By the defining property, for any $m$ the set $V(\F_{q^m})$ is in bijection to $C(\F_{q^m} \otimes_{\F_q} \F_{q^n})$.
  For $m \mid n$ we have $C(\F_{q^m} \otimes_{\F_q} \F_{q^n}) = C(\F_{q^n}^m) = C(\F_{q^n})^m = \emptyset$.

  Now let $m \geq 1$ with $\lcm(n, m) \geq 4n$.
  We have $\F_{q^m} \otimes_{\F_q} \F_{q^m} = \F_{q^{\lcm(n,m)}}^{nm/\lcm(n,m)}$.
  Then $C(\F_{q^m} \otimes_{\F_q} \F_{q^n}) = C(\F_{q^{\lcm(n,m)}}^{nm/\lcm(n,m)}) \neq \emptyset$ since $C(\F_{q^{\lcm(n,m)}}) \neq \emptyset$ by the defining property of $C$.
  This proves the desired property of $V$.
\end{proof}

\begin{remark}\label{rem:testVarietyComputable}
  For given $q$ and $m$, a variety $V$ as in the proposition can be effectively determined, simply by enumerating varieties, testing for points over small fields, and using the Hasse--Weil bound.
\end{remark}

\section{The construction}

Fix again a prime $p$.
We find an extension field $\K$ of $\F_p(t)$ satisfying a strong local-global principle, after Ershov.

We first fix some terminology.
A \emph{discrete} valuation is a Krull valuation whose value group is isomorphic to $\Z$, i.e.\ is given by a discrete valuation ring in the usual sense of commutative algebra.
We do not distinguish between valuations and their valuation rings, so in particular we identify equivalent valuations.
A valuation of $\F_p(t)$ is said to be \emph{above $\F_p[t]$} if its valuation ring contains $\F_p[t]$, i.e.\ if it is not the degree valuation of $\F_p(t)$.

\begin{proposition}
  There exists a countable regular field extension $\K/\F_p(t)$ together with a family $V$ of discrete valuations such that the following hold:
  \begin{enumerate}
  \item For every $v \in V$, the restriction of $v$ to $\F_p(t)$ is again a discrete valuation, which lies above $\F_p[t]$.
    Further, the extension $(\K,v)/(\F_p(t), v|_{\F_p(t)})$ is immediate, i.e.\ the residue fields $\K v$ and $\F_p(t)v|_{\F_p(t)}$ agree and a uniformiser for $v|_{\F_p(t)}$ remains a uniformiser for $v$.
  \item For every discrete valuation $v_0$ of $\F_p(t)$ above $\F_p[t]$ there is precisely one $v \in V$ prolonging $v_0$.
  \item Any $x \in \K$ is in the valuation ring of all but finitely many $v \in V$.
  \item If a geometrically integral variety $X/\K$ has a smooth $\K_v$-point for every $v \in V$ (where $\K_v$ is the henselisation), then it has a $\K$-point.
  \end{enumerate}
\end{proposition}
\begin{proof}
  This is a consequence of \cite[Theorem 3.6.3]{Ershov}, as we now explain.
  
  Let $V_0$ be the family of discrete valuation rings of $\F_p(t)$ above $\F_p[t]$.
  Then any two distinct members of $V_0$ are independent;
  $V_0$ is a near Boolean family in Ershov's sense since $\F_p[t]$ is an ``NB-ring'' \cite[Remark 2.5.1]{Ershov} and the valuation rings of the valuations in $V_0$ are precisely the localisations of $\F_p[t]$ at its maximal ideals \cite[Proposition 2.5.3]{Ershov};
  and the residue fields of $V_0$ are ``regularly closed at infinity'' \cite[Section 3.4, p.~172]{Ershov}, as they are finite fields with only finitely many of cardinality lower than a given bound, and so the desired property follows from the Lang-Weil bounds \cite[Theorem 7.7.1(iv)]{Poonen}.

  Thus by \cite[Theorem 3.6.3]{Ershov}, there exists a countable regular extension $\K/\F_p(t)$ with a family of valuation rings $V$ such that every valuation $v \in V$ lies over a valuation $v_0 \in V_0$, this induces a bijection between $V$ and $V_0$, and the extension of valued fields $(\K, v)/(\F_p(t), v_0)$ is immediate.
  In particular every $v \in V$ is discrete, and conditions (1) and (2) are satisfied.

  The bijection $V \to V_0$ is furthermore a homeomorphism with respect to the Zariski topologies on $V$ and $V_0$ (see \cite[Section 2.2]{Ershov}), which means that for all $x \in \K$ the (``Zariski closed'') set $C_x := \{ v \in V \colon v(x) < 0 \}$ is either finite or all of $V$ since the analogous statement holds in $\F_p(t)$.
  However, we cannot have $C_x = V$, since otherwise for a suitable element $b \in \F_p(t)^\times$ (a high power of a uniformiser for some valuation in $V_0$) the set $C_{xb}$ would be infinite but strictly contained in $V$, violating the homeomorphism property.
  Therefore the set $C_x$ is finite for all $x \in \K$.
  This gives condition (3).
  
  In addition, the family $V$ satisfies Ershov's ``arithmetic local-global principle'' LG\textsubscript{A}, and therefore also the ``geometric local-global principle'' LG\textsubscript{G} \cite[Proposition 3.2.5]{Ershov}, which gives our condition (4) for geometrically integral affine varieties $X/\K$.
  Now take an arbitrary geometrically integral variety $X/\K$, and let $X_0/\K$ be an affine dense open subvariety.
  If $X$ has a smooth $\K_v$-point for every $v \in V$, then the same holds for $X_0$: This is the ampleness of the henselian field $\K_v$ (see \cite[Corollary 3.1.6]{Ershov} and the surrounding discussion).
  Hence we have $\emptyset \neq X_0(\K) \subseteq X(\K)$ by the affine case, proving (4) in full generality.
\end{proof}

\begin{remarks}
  \begin{enumerate}
  \item Fields $\K$ as in the proposition are weak analogues of the ``surprising extensions of $\Q$'' considered in \cite{Ershov_surprising} (also variously translated as ``wonderful'' or ``amazing'' extensions).
    Note, however, that there also the place at infinity, i.e.\ the real place of $\Q$, is included.
  \item Since it plays no role in the sequel, we have not imposed the condition which is called maximality in \cite{Ershov_surprising}, i.e.\ that for every proper separable algebraic extension $L/\K$, some valuation in $V$ has no immediate extension to $L$.
    This condition can however always be added, see \cite[Proposition 4.4.3, Remark 4.4.3, Proposition 4.4.4]{Ershov}.
  \item Any non-trivial valuation $v$ of $\K$ not in $V$ always has separably closed henselisation, and hence poses no obstruction to the existence of rational points on varieties.
    This follows from \cite[Corollary 3.5.4]{Ershov} (there stated for boolean families of valuations, but the same proof works for near-boolean families with residue fields regularly closed at infinity).
    In particular, the family $V$ simply consists of all discrete valuations of $\K$.
  \item Instead of starting with the discrete valuations of $\F_p(t)$ above $\F_p[t]$, we could have worked with the coordinate ring of any irreducible smooth affine curve over $\F_p$ and its function field.
  \end{enumerate}
\end{remarks}

We henceforth fix a field $\K$ as in the proposition.

\begin{lemma}\label{lem:points-finite-ext}
  Let $L/\K$ be a finite separable extension.
  Let $X/\F_p$ be a smooth projective geometrically integral variety.
  Then $X(L) \neq \emptyset$ if and only if for every $v \in V$ and every prolongation $w$ of $v$ to $L$, $X$ has a point over the residue field $Lw$.
\end{lemma}
\begin{proof}
  Let $W$ be the family of prolongations of valuations in $V$ to $L$.
  By \cite[Proposition 3.4.1]{Ershov} (a Weil restriction argument), the same local-global principle as for $V$ holds for $W$.
  In particular, $X(L) \neq \emptyset$ if and only if $X$ has a point over all henselisations $L_w$, $w \in W$.

  Let $w \in W$.
  If $X$ has a point over the henselisation $L_w$, then it has a point over the residue field $Lw$, using that $X$ is projective (given homogenous coordinates of an $L_w$-point of $X$, clear denominators and reduce).\footnote{Using
    the valuative criterion of properness, it would suffice to assume that $X$ is proper instead of projective.}
  Conversely, if $X$ has a point over the residue field $Lw$, then it has a point over the henselisation $L_w$ since there exists an embedding $Lw \hookrightarrow L_w$
  (apply Hensel's Lemma to the minimal polynomial of a primitive element of $Lw$ over $\F_p$).
  Together with the local-global principle, this proves the statement.
\end{proof}

We next wish to find finite extensions $L/\K$ such that the discrete valuations of $L$ have prescribed residue fields.
This is achieved by the following lemmas.

\begin{lemma}\label{lem:prescribed-residue-fields-global}
  Let $S_1$, $S_2$ be two disjoint finite sets of prime numbers greater than $4$.
  There exists a cyclic extension $L_0/\F_p(t)$ of degree $4$ such that:
  \begin{enumerate}
  \item For every $l \in S_1$, there exists a discrete valuation of $L_0$ above $\F_p[t]$ with residue field $\F_{p^l}$.
  \item For every $l \in S_2$ and every $\F_{p^m}$ occurring as the residue field of a valuation of $L_0$, we have $\lcm(l,m) \geq 4l$.
  \end{enumerate}
\end{lemma}
\begin{proof}
  Let $L_0/\F_p(t)$ be a cyclic extension of degree $4$ in which each of the (non-zero) finitely many discrete valuations of $\F_p(t)$ with residue field $\F_{p^l}$, $l \in S_1$, is completely split, and each of the finitely many discrete valuations of $\F_p(t)$ with residue field $\F_{p^n}$, $S_1 \not\ni n \leq 4 \max(S_2)$, is inert.
  The existence of such an extension follows from the Grunwald--Wang Theorem \cite[Theorem 9.2.8]{NSW}, which allows the construction of abelian extensions of $\F_p(t)$ in which the decomposition behaviour of finitely many places is prescribed.
  The field $L_0$ satisfies the required properties.
\end{proof}

\begin{lemma}\label{lem:prescribed-residue-fields-wonderful}
  Let $S_1$, $S_2$ be two disjoint finite sets of prime numbers greater than $4$.
  Then there exists a cyclic extension $L/\K$ of degree $4$ such that conditions (1) and (2) from Lemma \ref{lem:prescribed-residue-fields-global} hold for $L$ (in place of $L_0$).
\end{lemma}
\begin{proof}
  Take $L_0/\F_p(t)$ as in Lemma \ref{lem:prescribed-residue-fields-global}, and let $L = \K L_0$ (free compositum, equivalently the tensor product $\K \otimes_{\F_p(t)} L_0$).
  For any discrete valuation $v$ of $L$, the restriction $w$ to $L_0$ is also a discrete valuation and we have the inclusion of residue fields $L_0w \subseteq Lv$.
  Thus condition (2) transfers from $L_0$ to $L$:
  Indeed, if $m_0 = [L_0w : \F_p]$ and $m = [Lv : \F_p]$, we have $m_0 \mid m$ and thus $4l \leq \lcm(l, m_0) \mid \lcm(l, m)$.
  
  On the other hand, for every discrete valuation $w$ of $L_0$ above $\F_p[t]$, the restriction $v_0$ to $\F_p(t)$ is again discrete, and the defining property of $\K$ affords a discrete valuation $v$ on $\K$ such that $(\K,v)/(\F_p(t), v_0)$ is immediate.
  In particular, $\K$ embeds into the completion $\widehat{\F_p(t)}_{v_0}$ over $\F_p(t)$.
  Thus $L$ embeds into the completion $\widehat{L_0}_w$ over $\F_p(t)$ since both $\K$ and $L_0$ have such an embedding and are linearly disjoint over $\F_p(t)$.
  Therefore $L$ carries a discrete valuation above $\F_p[t]$ with residue field $L_0w$.
  Thus condition (1) transfers from $L_0$ to $L$.
\end{proof}

We can now show that $\aleph_0$-saturated elementary extensions $\K^\ast$ of $\K$ have existentially undecidable finite extensions.

Recall (see for instance \cite[Definition 1.6.8]{Soare}) that a set of natural numbers $A$ is \emph{many-one reducible} to a set of natural numbers $B$ if there exists a computable function $f \colon \N \to \N$ such that for any $x \in \N$ we have $f(x) \in B$ if and only if $x \in A$.
This is a formalisation of the notion that membership in $A$ is no harder to decide than membership in $B$.
(A different formalisation is given by Turing reducibility, which is implied by many-one reducibility.)

By fixing a standard Gödel coding, we identify formulae of a given finite first-order language with natural numbers.
In particular, computability-theoretic terms such as decidability and many-one reducibility make sense for sets of formulae.
We generally work in the language of rings $\Lring = \{ +, -, \cdot, 0, 1 \}$.

\begin{theorem}\label{thm:undecidable-finite-extensions}
  Let $S$ be a set of prime numbers.
  Then any $\aleph_0$-saturated elementary extension $\K^\ast$ of $\K$ has a cyclic extension $L/\K^\ast$ of degree $4$ such that $S$ is many-one reducible to the existential theory of $L$.
  In particular, there exist cyclic extensions $L/\K^\ast$ of degree $4$ with undecidable existential theory.
\end{theorem}
\begin{proof}
  For every prime number $l$, let $V_l/\F_p$ be a variety as in Proposition \ref{prop:testVariety} (with $q=p$, $n=l$).
  We claim that we can choose $L$ such that for all primes $l > 4$, we have $V_l(L) \neq \emptyset$ if and only if $l \in S$.
  Since $V_l$ can be computed from $l$ by Remark \ref{rem:testVarietyComputable}, and $V_l(L) \neq \emptyset$ is straightforwardly translated into an existential sentence, this $L$ solves the problem.

  It thus remains to find $L$ satisfying the claim.
  Recasting the search for $L$ as the search for the coefficients of an irreducible polynomial of degree $4$ over $\K^\ast$ with a root generating $L$, saturation reduces us to finding, for every finite set of primes $S_1$ disjoint from $S$ and finite $S_2 \subseteq S$, an extension $L/\K$ of degree $4$ with $V_l(L) = \emptyset$ for $4 < l \in S_1$ and $V_l(L) \neq \emptyset$ for $4 < l \in S_2$.
  This problem is solved by Lemma \ref{lem:prescribed-residue-fields-wonderful}:
  Indeed, the field $L$ produced there satisfies the condition by Lemma \ref{lem:points-finite-ext} and the construction of the $V_l$.

  The ``in particular'' holds because if $S$ is an undecidable set, then the existential theory of $L$ cannot be decidable.
\end{proof}

\begin{remark}\label{rem:countable-instead-of-saturated}
  The passage to an elementary extension of $\K$ is due to the need to realise a certain type, given by the requirements for the coefficient tuple of an irreducible polynomial defining $L$.
  Given that this is only one type, a well-chosen countable elementary extension $\K^\ast$ of $\K$ (depending on $S$) would be sufficient in place of an $\aleph_0$-saturated one.
\end{remark}

\section{Existential decidability}

Let again $p$ be a prime number, $\K/\F_p(t)$ as in the last section.
In this section we prove that the existential theory of $\mathbb K$ is decidable.
\begin{lemma}
  Let $X/\F_p$ be a geometrically integral smooth affine variety.
  Then $X(\K) \neq \emptyset$ if and only if $X(\F_p\pow{s}) \neq \emptyset$.
\end{lemma}
\begin{proof}
  First observe that since $\K$ carries a discrete valuation with residue field $\F_p$ (for instance the prolongation in $V$ of the $t$-adic valuation of $\F_p(t)$), $\K$ embeds into $\F_p\pow{s}$, and therefore the existence of a $\K$-rational point of $X$ implies the existence of an $\F_p\pow{s}$-rational point.

  Suppose conversely that $X$ has an $\F_p\pow{s}$-rational point.
  Then it has a point over the henselisation $\F_p(s)_s$ at the $s$-adic valuation, since the fields $\F_p\pow{s}$ and $\F_p(s)_s$ have the same existential theory by \cite[Corollary 7.2]{AnscombeFehm_existential-old} (or by \cite[Theorem 5.9]{Kuhlmann_tame}).
  Therefore $X$ has a rational point over the henselisation $\K_v$ for every $v$, since every such henselisation embeds $\F_p(s)_s$ (sending $s$ to a uniformiser).
  Now $X(\K) \neq \emptyset$ follows from the local-global principle.
\end{proof}

The following general lemma reduces the existential theory of a field to information about which smooth affine varieties have rational points.
This may well have appeared elsewhere in the literature, but I am unaware of a reference.
As usual, given a field $F$, the language $\Lring(F)$ is simply the expansion of $\Lring$ by constants for the elements of $F$.
In particular, every extension $E/F$ is naturally an $\Lring(F)$-structure.
\begin{lemma}\label{lem:reduction-geom-int-smooth}
  Let $F$ be a field, and $E_1/F$, $E_2/F$ two regular extensions.
  Assume that for every geometrically integral smooth affine $F$-variety $X$ we have $X(E_1) \neq \emptyset$ if and only if $X(E_2) \neq \emptyset$.
  Then the existential $\Lring(F)$-theories of $E_1$ and $E_2$ agree.
\end{lemma}
\begin{proof}
  By standard reductions (disjunctive normal form, elimination of inequalities) it suffices to show that for any $f_1, \dotsc, f_k \in F[X_1, \dotsc, X_n]$, the $f_i$ have a common zero in $E_1$ if and only if they have a common zero in $E_2$.
  In other words, we must show that for every affine $F$-variety $X$ we have $X(E_1) \neq \emptyset$ if and only if $X(E_2) \neq \emptyset$.

  By passing to the reduction of $X$ is necessary, and using that for every reduced variety the regular locus is open and not empty \cite[Corollary 12.52(2)]{GoertzWedhorn}, we can write $X$ as a union of finitely many regular integral affine locally closed subvarieties.
  In other words, it suffices to consider integral regular affine $X$.

  If $X$ is not geometrically integral, then $X(E_1) = \emptyset = X(E_2)$:
  Indeed, the base-changed varieties $X_{E_1}/E_1$ and $X_{E_2}/E_2$ are regular \cite[Proposition 6.7.4]{EGA-IV-2}, integral \cite[Corollary 5.56(3)]{GoertzWedhorn}, but not geometrically integral, and therefore they have no rational points (see for instance \cite[Lemma 10.1]{Poonen_smooth-projective}).

  Thus let us assume that $X$ is geometrically integral.
  Then the smooth locus $X_{\mathrm{sm}} \subseteq X$ is dense open \cite[Theorem 6.19, Remark 6.20(ii)]{GoertzWedhorn}.
  Any $E_1$-rational point on $X$ is necessarily smooth \cite[Proposition 17.15.1]{EGA-IV-4}, i.e.\ is an $E_1$-rational point on the geometrically integral smooth variety $X_{\mathrm{sm}}$.
  By the assumption applied to the open affine subvarieties of $X_{\mathrm{sm}}$, we must then also have an $E_2$-rational point on $X_{\mathrm{sm}}$ and therefore on $X$.
  By symmetry, this shows that $X(E_1) \neq \emptyset$ if and only if $X(E_2) \neq \emptyset$, as desired.
\end{proof}

\begin{proposition}
  The existential theory of $\K$ agrees with the existential theory of $\F_p\pow{s}$.
  In particular, it is decidable.
\end{proposition}
\begin{proof}
  The first statement follows from the two preceding lemmas (take $F=\F_p$, $E_1 = \K$, $E_2 = \F_p\pow{s}$).
  The ``in particular'' is \cite[Corollary 7.5]{AnscombeFehm_existential-old}.
\end{proof}

\begin{corollary}\label{cor:quadratic-pair}
  There exists an existentially decidable field $K$ of characteristic $p$ with an existentially undecidable separable quadratic extension.
  We can choose $K$ such that the relative algebraic closure of $\F_p$ in $K$ is finite.
\end{corollary}
\begin{proof}
  Let $\K^\ast$ be an $\aleph_0$-saturated elementary extension of $\K$, and let $L/\K^\ast$ be a cyclic extension of degree $4$ which is existentially undecidable (Theorem \ref{thm:undecidable-finite-extensions}).
  Let $L_0/\K^\ast$ be the unique quadratic intermediate field.
  Since $\K$ is regular over $\F_p(t)$, the prime field $\F_p$ is relatively algebraically closed in $\K$ and thus in $\K^\ast$.
  Hence the relative algebraic closure of $\F_p$ in $L$ is finite.
  Since $\K^\ast$ is existentially decidable as $\K$ is, either $L_0/\K^\ast$ or $L/L_0$ is a pair of fields as desired.
\end{proof}

This proves Theorem \ref{thm:intro-Kesavan} from the introduction.
As mentioned previously, the analogue in characteristic $0$ was established in \cite[Theorem 3.3.1]{Kesavan}, with a similar technique.
There the full first-order theory of the base field is decidable, so the result is stronger than ours (inspection of the proof yields that the quadratic extension still has undecidable existential theory).
Decidability of the full first-order theory seems out of reach in positive characteristic with the current method, as our understanding of the model theory of valued fields is insufficient.

Conditionally on a conjecture related to resolution of singularities, we can establish a slightly stronger decidability result in the language $\Lring(\F_p(t))$.
Here we fix a natural coding of $\F_p(t)$ (specifically, a coding witnessing the computability of the field $\F_p(t)$) to identify $\Lring(\F_p(t))$-formulae with natural numbers.
Since every element of $\F_p(t)$ is quantifier-freely $\Lring$-definable in terms of the constant $t$, instead of $\Lring(\F_p(t))$ we could equivalently work in the expansion of $\Lring$ by a single constant symbol for $t$.
\begin{lemma}
  Assume the consequence (R4) of local uniformisation from \cite{ADF_existential}.
  Then there is an algorithm which, given as input $k, n > 0$ and polynomials $f_1, \dotsc, f_k \in \F_p(t)[X_1, \dotsc, X_n]$ such that the affine variety described by the $f_i$ is geometrically integral and smooth over $\F_p(t)$, decides whether the variety has a $\K$-rational point, i.e.\ whether the $f_i$ have a common zero in $\K$.
\end{lemma}
\begin{proof}
  Let $V_0$ be the set of discrete valuations of $\F_p(t)$ above $\F_p[t]$.
  By \cite[Remark 7.7.3]{Poonen} (a combination of the Lang--Weil bounds and Hensel's Lemma) one can effectively determine a finite subset $S \subseteq V_0$ such that for all $v \in V_0 \setminus S$ the $f_i$ have a common zero in the completion $\widehat{\F_p(t)}_v$,
  and therefore in the henselisation $\F_p(t)_v$ since $\F_p(t)_v$ is existentially closed in $\widehat{\F_p(t)}_v$ \cite[Theorem 5.9]{Kuhlmann_tame}.

  In order to decide whether the $f_i$ have a common zero in $\K$, by the local-global principle it therefore suffices to decide whether they have a common zero in the henselisation $\K_v$ for each discrete valuation $v$ of $\K$ above one of the valuations in $S$.
  The decidability of this problem (under the assumption (R4)) follows from \cite[Theorem 4.12]{ADF_existential}.
  Indeed, for each such $v$, the henselisation $\K_v$ (with the canonical valuation) is an immediate extension of $\F_p(t)$ (with the restricted valuation), and therefore its universal/existential $\Lring(\F_p(t))$-theory is formally entailed by the first-order axioms expressing that it is a henselian valued field extending $(\F_p(t), v|_{\F_p(t)})$ with the same residue field and uniformiser.
\end{proof}

\begin{proposition}
  Assume (R4).
  Then the existential $\Lring(\F_p(t))$-theory of $\K$ is decidable.
\end{proposition}
\begin{proof}
  Consider the $\Lring(\F_p(t))$-theory $T$ given by the following system of axioms:
  \begin{enumerate}
  \item the field axioms;
  \item the quantifier-free diagram of $\F_p(t)$;
  \item for each irreducible polynomial $f \in \F_p(t)[X]$ the sentence $\forall x (f(x) \neq 0)$;
  \item for any finite list of polynomials $f_1, \dotsc, f_k \in \F_p(t)[X_1, \dotsc, X_n]$ describing a geometrically integral smooth affine $\F_p(t)$-variety, an axiom asserting that the $f_i$ have a common zero if this is the case in $\K$, and otherwise an axiom asserting that they do not have a common zero.
  \end{enumerate}
  We claim that this system of axioms is computably enumerable.
  This is clear for the field axioms, follows from the computability of $\F_p(t)$ for the quantifier-free diagram, and from the existence of a splitting algorithm for $\F_p(t)$ for the third point.
  For the fourth point, this is essentially the preceding lemma and the observation that it is decidable whether a system of polynomials defines a geometrically integral smooth variety (for instance by Gröbner basis techniques and the Jacobian criterion).

  The models of $T$ are field extensions $E$ of $\F_p(t)$ in which $\F_p(t)$ is relatively algebraically closed, i.e.\ which are regular over $\F_p(t)$, and such that the same geometrically integral smooth affine $\F_p(t)$-varieties have rational points in $E$ as in $\K$.
  By Lemma \ref{lem:reduction-geom-int-smooth}, the theory $T$ is therefore complete for universal and existential $\F_p(t)$-sentences, i.e.\ for any existential $\F_p(t)$-sentence, $T$ entails either the sentence or its negation.
  A proof calculus therefore gives a decision procedure for existential consequences of $T$, which proves the claim since $\K \models T$.
\end{proof}

\section{An existentially undecidable complete valued field}

Let $p$ be a fixed prime.
We prove the following (stated as Theorem \ref{thm:intro-valued} in the introduction):
\begin{theorem}\label{thm:valued-example}
  There exists a complete discretely valued field $(E, v)$ with $\kar E = 0$, $\kar Ev = p$, such that the existential theory of $Ev$ is decidable, but the existential theory of $E$ is undecidable.
  We can furthermore choose $E$ such that the set of one-variable polynomials in $\Q[X]$ with a zero in $E$ is decidable.
\end{theorem}
\begin{proof}
  By Corollary \ref{cor:quadratic-pair}, we may select
  an existentially decidable field $K$ of characteristic $p$ with an existentially undecidable separable quadratic extension $L$,
  such that furthermore the relative algebraic closure of $\F_p$ in $K$ is finite.
  
  There is an element $\alpha \in L$ with $L = K(\alpha)$ and $a := \alpha^2 - \alpha \in K$.
  (We use this equation instead of $a = \alpha^2$ to handle all characteristics simultaneously.)
  
  Let $(F, v)$ be the unique complete discretely valued field of characteristic $0$ with residue field $K$ and uniformiser $p$.
  (See for instance \cite[Theorem 2.10 and Corollary 6.6]{AnscombeJahnke_Cohen} for the (classical) existence and uniqueness of such $(F,v)$ in terms of the valuation ring, although we do not in fact need the uniqueness.)
  Let $b \in F$ be a lift of $a$, and  
  set $E = F(\sqrt{p (1 + 4b)})$.
  We continue to write $v$ for the unique prolongation to finite extensions of $F$, in particular to $E$.

  We claim that $(E, v)$ is as desired.
  Note first that $v(1+4b) = 0$:
  this is clear if $p=2$, and holds for odd $p$ since otherwise $a = -1/4$ and so the polynomial $X^2 - X - a$ would be reducible in $K$.
  Therefore the extension $E$ is obtained by adjoining to $F$ a square root of the uniformiser $p(1+4b)$, and is thus totally ramified.
  In particular $Ev = Fv = K$, which is existentially decidable.

  On the other hand, the field $E(\sqrt{p})$ contains the element $\sqrt{1 + 4b}$, and therefore a root of the polynomial $X^2 - X - b$, so the residue field $E(\sqrt{p})v$ must be $K(\alpha) = L$.
  In the complete discretely valued field $(E(\sqrt{p}), v)$ both the valuation ring $\mathcal{O}_v$ and its maximal ideal $\mathfrak{m}_v$ are existentially $\Lring$-definable (without parameters), since for a natural number $n > 2$ coprime to $p$ a well-known application of Hensel's Lemma shows that
  \[ \mathcal{O}_v = \{ x \in E(\sqrt{p}) \colon \exists y (1 + px^n = y^n) \}, \quad \mathfrak{m}_v = \{ x \in E(\sqrt{p}) \colon \exists y (1 + x^n/p = y^n)\} .\]
  Therefore the residue field $L$ is (parameter-freely) existentially interpretable in $E(\sqrt{p})$ (i.e.\ we have an interpretation satisfying the property of \cite[Theorem 5.3.2, Remark 3]{Hodges_longer}), so the existential $\Lring$-theory of $E(\sqrt{p})$ is undecidable.
  (See \cite[Theorem 5.3.2, Remark 4]{Hodges_longer} for generalities on transfer of decidability under interpretations.)
  Consequently, the existential $\Lring$-theory of $E$ is likewise undecidable, since $E(\sqrt{p})$ is quantifier-freely interpretable in $E$.

  Lastly, consider the subfield $\Q_p \subseteq F$ (given as the topological closure of the subfield $\Q$).
  Since the relative algebraic closure of $\F_p$ in $Fv = K$ is finite and $(F,v)$ has uniformiser $p$, the fundamental equality for algebraic extensions of $\Q_p$ (see for instance \cite[Proposition 1.4.6]{Ershov}) shows that the relative algebraic closure of $\Q_p$ in $F$ is a finite extension of $\Q_p$, and therefore the same holds in $E$.
  Thus the algebraic part of $E$ is the same as the algebraic part of a local field of characteristic zero.
  The local fields of characteristic zero have decidable first-order theory \cite[Corollary 5.3]{PrestelRoquette}, so in particular it is decidable whether a given polynomial in $\Q[X]$ has a zero in $E$.
\end{proof}

\begin{remark}\label{rem:decidable-algebraic-part}
  The condition that the set of polynomials in $\Q[X]$ with a zero in $E$ be decidable is occasionally phrased as $E$ having ``decidable algebraic part'', in the sense that it allows to decide which elements of an algebraic closure of $\Q$ lie in $E$ (up to conjugacy).
  There is however a certain ambiguity in this expression, as it may also be understood to assert that the field $E \cap \overline{\Q}$ is decidable, i.e.\ has decidable full first-order theory, which is a stronger condition.
  Our proof of Theorem \ref{thm:valued-example} shows that even this stronger condition is satisfied, since the algebraic part of a local field is an elementary substructure and therefore shares its (decidable) first-order theory \cite[Theorem 3.4 and Theorem 5.1]{PrestelRoquette}.
\end{remark}

\begin{remark}\label{rem:question-AF}
  In \cite[Remark 7.6]{AnscombeFehm_existential-old} it was asked whether there exists an existentially undecidable henselian valued field of mixed characteristic with existentially decidable residue field and pointed value group (i.e.\ value group with a constant for the value of $p$).
  Theorem \ref{thm:valued-example} provides an example for this phenomenon, as even the full first-order theory of the value group $\Z$ is decidable \cite[Theorem 3.3.8]{Hodges_longer} (and expanding by a constant symbol for $v(p)$ does not change this, since any constant in $\Z$ is definable).

  It was previously pointed out that there must exist (non-discrete) examples of valued fields with the desired property in \cite[Remark 3.6.9]{Kartas_undecidability-tilting}, using an inexplicit counting argument.
  However, the reason for existential undecidability of the examples there is due to it not being decidable which one-variable polynomials over $\Q$ have roots, unlike in our example.

  The algebraic part has been known for some time as an obstruction in the model theory of henselian valued fields of mixed characteristic, see for instance \cite[Corollary 1.6]{AnscombeKuhlmann_notes} and \cite[Remark 7.4]{AnscombeFehm_existential-old}.
  Our theorem shows that the obvious attempt to repair the failure of the decidability statement \cite[Corollary 7.5]{AnscombeFehm_existential-old} in mixed characteristic, by requiring a decidable axiom scheme describing the algebraic part, still fails, even in the case of value group $\Z$.
\end{remark}

\end{document}